\documentclass[12pt,amsymb,fullpage]{amsart}
\usepackage{amssymb,amscd,pstricks}

\newtheorem{theorem}{Theorem}[section]
\newtheorem{defn}[theorem]{Definition}

\newtheorem{lemma}[theorem]{Lemma}

\newtheorem{eple}[theorem]{Example}
\newtheorem{rmk}[theorem]{Remarks}
\newtheorem{dsc}[theorem]{Discussion}
\newtheorem{nota}[theorem]{Notation}

\newsavebox{\indbin}
\savebox{\indbin}{\begin{picture}(0,0)
\newlength{\gnu}
\settowidth{\gnu}{$\smile$} \setlength{\unitlength}{.5\gnu}
\put(-1,-.65){$\smile$} \put(-.25,.1){$|$}
\end{picture}}

\newcommand{\be}{\begin{enumerate}}
\newcommand{\bd}{\begin{defn}}
\newcommand{\bt}{\begin{theorem}}
\newcommand{\bl}{\begin{lemma}}
\newcommand{\ee}{\end{enumerate}}
\newcommand{\ed}{\end{defn}}
\newcommand{\et}{\end{theorem}}
\newcommand{\el}{\end{lemma}}

\begin{document}
\title{Decay Rates for Cusp Functions}
\author{Tristram de Piro}
\address{Mathematics Department, The University of Exeter, Exeter}
 \email{tdpd201@exeter.ac.uk}
\thanks{}
\begin{abstract}
We make some observations on the decay rates of the Fourier coefficients of cusp functions.
\end{abstract}
\maketitle

\begin{defn}
\label{coefficient}
We let $H(0,1)$ denote the restrictions of holomorphic functions on the open disc $D(0,1+\epsilon)$, for some $\epsilon>0$, and $C^{\infty}(S^{1})$, the set of smooth, complex-valued functions on the unit circle $S^{1}$. For $f\in H(0;1)$, with power series expansion $f(z)=\sum_{n\in\mathcal{Z}_{\geq 0}}a_{n}z^{n}$, we let $\overline{f}(n)=a_{n}$, for $n\in\mathcal{Z}_{\geq 0}$, and $f_{1}=f|_{{\partial} D(0;1)}$. For $g\in C^{\infty}(S^{1})$, and $n\in\mathcal{Z}$, we let;\\

$\hat{g}(n)={1\over 2\pi}\int_{0}^{2\pi}g(\theta)e^{-i n\theta}d\theta$\\

\noindent denote the $n$'th Fourier coefficient.

\end{defn}

\begin{lemma}
With notation as in \ref{coefficient}, if $f\in H(0,1)$, and $n\in\mathcal{Z}_{\geq 0}$, we have $\overline{f}(n)=\hat{f_{1}}(n)$.
\end{lemma}

\begin{proof}
By the definition of residues for holomorphic function, Cauchy's residue theorem, see \cite{P}, and Fourier series, see \cite{SS}, we have, for $n\in\mathcal{Z}_{\geq 0}$, that;\\

$\overline{f}(n)={1\over 2\pi i}\int_{\partial D(0,1)}{f(z)\over z^{n+1}}dz$\\

$={1\over 2\pi i}\int_{0}^{2\pi}{f(e^{i\theta})\over e^{i(n+1)\theta}}ie^{i\theta}d\theta$\\

$={1\over 2\pi}\int_{0}^{2\pi}{f(e^{i\theta})\over e^{i n\theta}}d\theta=\hat{f_{1}}(n)$\\

\end{proof}

\begin{rmk}
\label{scaling}
By rescaling, we obtain a similar result for $0<\delta<1$ and $f\in H(0,1-\delta)$, $f_{1-\delta}\in C^{\infty}(S^{1})$, where $f_{1-\delta}(z)=f((1-\delta)z)$, for $|z|=1$. Namely;\\

$\overline{f}(n)={1\over (1-\delta)^{n}}\hat{f}_{1-\delta}(n)$.

\end{rmk}

\begin{defn}
\label{cusp}
We let $\mathcal{H}_{+}=\{z\in\mathcal{C}:Im(z)>0\}$ denote the upper half plane and $H(\mathcal{H}_{+})$ denote the set of holomorphic functions on $\mathcal{H}_{+}$. We let $Cusp(\mathcal{H}_{+})\subset H(\mathcal{H}_{+})$ denote holomorphic functions on $\mathcal{H}_{+}$, satisfying the additional symmetry condition;\\

$(i)$. $g(z+1)=g(z)$, for $z\in\mathcal{C}$.\\

and the cusp condition;\\

$(ii)$. $lim_{Im(z)\rightarrow\infty}g(z)=0$ (we assume the limit is uniform in $Re(z)$)\\

\end{defn}

\begin{rmk}
\label{forms}
Observe that cusp forms , see \cite{FK}, are special examples of cusp functions. The map $\Phi=exp(2\pi iz):\mathcal{H}_{+}\rightarrow D(0,1)$ is holomorphic, and, taking a principal branch of the logarithm $\Gamma={log(z)\over 2\pi i}:(D(0,1)\setminus{[0,1)})\rightarrow({\mathcal{H}}_{+}\cap U)$, where $U=\{z\in\mathcal{H}_{+}:0<z<1\}$, we obtain a holomorphic function $\Gamma^{*}(g)$ on $(D(0,1)\setminus{[0,1)})$.
The condition $(i)$ ensures that $\Gamma^{*}(g)$ extends uniquely to a holomorphic function $f_{0}$ on the annulus ${D(0,1)\setminus\{0\}}$, such that $\Phi^{*}(f_{0})=g$. Using a Laurent's Theorem and the condition $(ii)$, $f_{0}$ extends uniquely to a holomorphic function $f$ on $D(0,1)$ with $f(0)=0$. Taking the power series expansion $f(z)=\sum_{n\in\mathcal{Z}_{\geq 0}}a_{n}z^{n}$ on $D(0,1)$, we obtain a convergent expansion $g(z)=\sum_{n\in\mathcal{Z}_{\geq 0}}a_{n}exp(2\pi inz)$ on $\mathcal{H}_{+}$. Similarly to before, for $n\in\mathcal{Z}_{\geq 0}$, we let $\overline{g}(n)=a_{n}$.

\end{rmk}

\begin{lemma}
\label{decay}
If $g\in Cusp(\mathcal{H}_{+})$, then, for any $m\in\mathcal{Z}_{>0}$, there exist constants $C_{m}\in\mathcal{R}_{\geq 0}$ and $D_{m}\in\mathcal{Z}_{>0}$, such that;\\

$|\overline{g}(n)|\leq C_{m}n^{-m}$, for all $n\geq D_{m}$\\
\end{lemma}

\begin{proof}
Let $f$ denote the holomorphic function on $D(0,1)$, obtained in Remark \ref{forms}. For any $0<\delta<1$, we have that $f\in H(0,1-\delta)$. By Remarks \ref{scaling}, we have that $\overline{f}(n)={1\over (1-\delta)^{n}}\hat{f}_{1-\delta}(n)$, $(*)$. As ${f}_{1-\delta}\in C^{\infty}(S^{1})$, by a simple extension of Corollary 2.4 in \cite{SS}, we have that, for any $m\in\mathcal{Z}_{>0}$, there exist constants $C_{m}\in\mathcal{R}_{\geq 0}$ and $B_{m}\in\mathcal{Z}_{>0}$, such that;\\

$|\hat{f}_{1-\delta}(n)|\leq A_{m}n^{-m}$, for all $n\geq B_{m}$\\

Then, by $(*)$ and Remarks \ref{forms}, we have;\\

$|\overline{g}(n)|=|\overline{f}(n)|={1\over (1-\delta)^{n}}|\hat{f}_{1-\delta}(n)|\leq C_{m}n^{-m}$, for all $n\geq B_{m}$\\

where $C_{m}=A_{m}n^{-m}$.

 \end{proof}

 \begin{rmk}
 This result seems to improve significantly on the result of the Ramunajan-Petersson conjecture, that, for cusp forms of weight $k\geq 2$, $|\overline{g}(n)|=O(n^{{k-1\over 2}+\gamma})$, $\gamma>0$, $n\in\mathcal{Z}_{\geq 0}$. The result also easily generalises to holomorphic functions on $\mathcal{H}_{+}$, satisfying just the symmetry condition $(i)$, in Definition \ref{cusp}, for $n\in\mathcal{Z}$, with $|n|\geq B_{m}$.
 \end{rmk}

\end{document}